\documentclass[12pt]{article}
\usepackage{amssymb}
\usepackage[all,arc]{xy}
\usepackage{mathrsfs}
\usepackage{amsfonts}
\usepackage{extarrows}
\usepackage{amsmath,amsthm,amsfonts,amssymb,amscd,color}

\allowdisplaybreaks[4]
\newtheorem{defn}{Definition}[section]
\newtheorem{them}[defn]{Theorem}
\newtheorem{lem}[defn]{Lemma}

\newtheorem{cor}[defn]{Corollary}

\newtheorem{prop}[defn]{Proposition}

\newtheorem{rem}[defn]{Remark}

\numberwithin{equation}{section}
\newenvironment {Proof} {\noindent {\bf Proof.}}{\quad $\square$\par\vspace{3mm}}

\begin{document}
\title{Sharp upper and lower bounds for the spectral radius of a  nonnegative
weakly irreducible tensor and its applications\footnote{The research is supported by the National Natural Science Foundation of China
(Grant No. 11571123), the Guangdong Provincial Natural Science Foundation (Grant No. 2015A030313377)
and Guangdong Engineering Research Center for Data Science.
}}
\author{ Lihua You$^{a,}$\footnote{{\it{Email address:\;}}ylhua@scnu.edu.cn.}
 \qquad Xiaohua Huang$^{a,}$\footnote{{\it{Email address:\;}}983920941@qq.com. }
 \qquad Xiying Yuan$^{b,}$\footnote{{\it{Corresponding author:\;}}xiyingyuan2007@hotmail.com. }}
\vskip.2cm
\date{{\footnotesize
$^{a}$School of Mathematical Sciences, South China Normal University, Guangzhou, 510631, P.R. China\\
$^{b}$Department of Mathematics, Shanghai University, Shanghai, 200444,  P.R. China\\
}}
\maketitle

\noindent {\bf Abstract } In this paper, we obtain the sharp upper and lower
bounds for the spectral radius of a nonnegative weakly irreducible tensor. We
also apply these bounds to the adjacency spectral radius and signless Laplacian
spectral radius of a uniform hypergraph.

\section{Introduction}

\hskip.6cm In recent years, the study of tensors and the spectra of tensors (and hypergraphs) with their various applications has attracted
extensive attention and interest, since the work of L. Qi (\cite{2005Q}) and L.H. Lim (\cite{2005L}) in 2005.

Denote by  $[n]=\{1, \ldots, n\}$. As is in  \cite{2005Q}, an order $m$ dimension $n$ tensor $\mathbb{A}= (a_{i_1i_2\ldots i_m})_{1\le i_j\le n \hskip.2cm (j=1, \ldots, m)}$
over the complex field $\mathbb{C}$ is a multidimensional array with all entries
$a_{i_1i_2\ldots i_m}\in\mathbb{C}\,, i_1, \ldots, i_m\in [n].$
A tensor  $\mathbb{A}= (a_{i_1i_2\ldots i_m})$ is called a nonnegative tensor if all of its entries $a_{i_1i_2\ldots i_m}$ are nonnegative.
Let  $X=(x_1,\ldots, x_n)^T\in \mathbb{C}^n, $ $X^{[r]}=(x_1^r, x_2^r, \ldots, x_n^r)^T$,
and $\mathbb{A}X^{m-1}$ be  a vector $\in \mathbb{C}^n$ whose $i$-th component is defined as the following:
$$
(\mathbb{A}X^{m-1})_i=\sum\limits_{i_2, \ldots, i_m=1}^na_{ii_2\ldots i_m}x_{i_2}\ldots x_{i_m}.
$$
Then  a number $\lambda\in \mathbb{C}$  is called an eigenvalue of $\mathbb{A}$ if there exists a nonzero vector  $X\in \mathbb{C}^n$ such that
\begin{equation}\label{eq12}
\mathbb{A}X^{m-1}=\lambda X^{[m-1]},
\end{equation}
and in this case, $X$ is called an eigenvector of $\mathbb{A}$ corresponding to eigenvalue  $\lambda$ (see \cite{ 2008C, 2005Q, 2014Qi}).

Recently, Shao \cite{2013S} defined the general product of two $n$-dimensional tensors as follows.

\begin{defn}{\rm (\cite{2013S})} \label{defn11}
Let $\mathbb{A}$ {\rm (}and $\mathbb{B}${\rm)} be an order $m\ge2$ {\rm (}and $k\ge 1${\rm)}, dimension $n$ tensor, respectively.
Define the general product  $\mathbb{A}\cdot\mathbb{B}$ (sometimes simplified as $\mathbb{A}\mathbb{B}$),  to be the following tensor $\mathbb{D}$ of order $(m-1)(k-1)+1$ and dimension $n$:
\begin{equation}\label{F-Shao}
 d_{i\alpha_1\ldots\alpha_{m-1}}=\sum\limits_{i_2, \ldots, i_m=1}^na_{ii_2\ldots i_m}b_{i_2\alpha_1}\ldots b_{i_m\alpha_{m-1}}
\end{equation}
where $i\in[n], \, \alpha_1, \ldots, \alpha_{m-1}\in[n]^{k-1}$.
\end{defn}

The tensor product  is a generalization of the usual matrix product, and satisfies  a very useful property:
the associative law (\cite{2013S}, Theorem 1.1). 
In this paper, all the tensor product obey Formula \ref{F-Shao}. According to Formula \ref{F-Shao}, the former $\mathbb{A}X^{m-1}$ is equal to the product $\mathbb{A}X$, i.e.,
$$
(\mathbb{A}X)_i=\sum\limits_{i_2, \ldots, i_m=1}^na_{ii_2\ldots i_m}x_{i_2}\ldots x_{i_m}.
$$

Now we recall some definitions and notations of matrices and graphs.

A square matrix $A$ of order $n$ is reducible if there exists a
permutation matrix $P$ of order $n$ such that:
$$PAP^T=\left(%
\begin{array}{cc}
  B & 0 \\
  D & C \\
\end{array}%
\right)$$ where $B$ and $C$ are square non-vacuous matrices.
$A$ is irreducible if it is not reducible. 

Let $D=(V,E)$ denote  a digraph on $n$ vertices.
A  $u\rightarrow v$ {\it walk} in $D$ is a sequence
of vertices $u, u_1,\ldots, u_k=v$ and a sequence of arcs $e_1=(u,u_1),e_2=(u_1,u_2),
\ldots, e_k=(u_{k-1},v)$, where the vertices and the arcs are not necessarily distinct.
A {\it path} is a walk with distinct vertices.
A digraph $D$  is said to be {\it strongly connected} if there
exists a path from $u$ to $v$ for all $u,v\in V$.

Let $A=(a_{ij})$ be a nonnegative square matrix of order $n$. The
{\it associated digraph} $D(A)=(V,E)$ of $A$ (possibly with loops) is defined
to be the digraph with vertex set $V=\{1,2,\ldots,n\}$ and arc set $E=\{(i,j)\hskip.1cm | \hskip.1cm a_{ij}> 0\}$.

The following well-known  theorem   gives the relationship between irreducibllity and strongly connectedness.

\begin{them}{\rm (\cite{1979A})} \label{thm13}
 A nonnegative matrix $A$ is irreducible if and only if its associated directed graph $D(A)$ is strongly connected.
\end{them}

In \cite{Friedland2013, 2011YY}, the weak irreducibility of tensors was defined and studied.

\begin{defn}{\rm (\cite{Friedland2013, 2011YY})}\label{defn13}
Let $\mathbb{A}$ be an  order $m$  dimensional $n$ tensor (not necessarily nonnegative).
If there exists a nonempty proper subset $I$ of the set $[n]$, such that
\begin{equation}\label{eq101}
a_{i_1i_2\ldots i_m}=0 \hskip.2cm (\forall  i_1\in I \mbox{ and  at least one of } i_2,\ldots, i_m\not\in I),
\end{equation}
 then $\mathbb{A}$ is called weakly reducible (or sometimes $I$-weakly reducible).
 If $\mathbb{A}$ is not weakly reducible, then $\mathbb{A}$ is called weakly irreducible.
\end{defn}

The following statement is an alternative explanation of weak irreducibility.

\begin{defn}{\rm (\cite{Friedland2013}\cite{2014S})} \label{defn14}
Suppose that $\mathbb{A}=(a_{i_1i_2\ldots i_m})_{1\le i_j\le n \hskip.2cm (j=1, \ldots, m)}$ is a nonnegative tensor of order $m$ and dimension $n$.
 We call a nonnegative matrix $G(\mathbb{A})$ the representation associated matrix to the nonnegative tensor $\mathbb{A}$,
if the ($i$, $j$)-th entry of  $G(\mathbb{A})$ is defined to be the summation of  $a_{ii_2\ldots i_m}$ 
with indices $\{i_2,\ldots i_m\}\ni j$.
 We call the tensor $\mathbb{A}$ weakly reducible if its representation $G(\mathbb{A})$ is a reducible matrix.
\end{defn}

By Definition \ref{defn14} and Theorem \ref{thm13}, we obtain the following proposition.

\begin{prop}\label{prop15}
Let $\mathbb{A}$ be a nonnegative  tensor of order $m$ and dimension $n$,
$G(\mathbb{A})$ be the representation associated matrix to  $\mathbb{A}$,
and $D(G(\mathbb{A}))$ be the associated directed graph of $G(\mathbb{A})$.
Then the following three conditions are equivalent:

\noindent {\rm (i).}  $\mathbb{A}$ is weakly irreducible.

\noindent {\rm (ii).}  $G(\mathbb{A})$  is  irreducible.

\noindent {\rm (iii).}  $D(G(\mathbb{A}))$ is strongly connected.
\end{prop}

The spectral radius of $\mathbb{A}$ is defined as
$$\rho(\mathbb{A})=\max\{|\lambda|:\lambda ~\mbox{ is an eigenvalue of }\mathbb{A}\}.$$

 \begin{lem}\label{P-F theorem}
 Let $\mathbb{A}$ be a nonnegative tensor.

\noindent{\rm (1). (\cite{2010Y}, Lemma 5.5) } If some eigenvalue of $\mathbb{A}$ has a positive eigenvector corresponding to it,
 then this eigenvalue must be $\rho(\mathbb{A})$.

\noindent{\rm (2). (\cite{Friedland2013}) } If $\mathbb{A}$ is weakly irreducible, then $\rho(\mathbb{A})$ has a positive eigenvector.
\end{lem}

  Let $\mathbb{A}$ be a tensor of order $m$ and dimension $n$. The $i$-th row sum
 of $\mathbb{A}$ is defined as $$r_i(\mathbb{A})=\sum\limits_{i_2,\ldots,i_m=1}^{n}a_{ii_2\ldots i_m}.$$

\begin{lem}{\rm (\cite{2010Y})} \label{lem18}
  Let $\mathbb{A}$ be a nonnegative tensor of dimension $n$. We have
\begin{equation}\label{eq13}
\min\limits_{1\leq i\leq n}r_i(\mathbb{A})\leq \rho(\mathbb{A}) \leq \max\limits_{1\leq i\leq n}r_i(\mathbb{A}).
\end{equation}
Moreover, if $\mathbb{A}$ is weakly irreducible, then one of the equalities in
(\ref{eq13}) holds if and only if $r_i(\mathbb{A})=\ldots=r_n(\mathbb{A})$.
\end{lem}
\begin{defn}{\rm (\cite{2013S})} \label{defn19}
Let $\mathbb{A}$ and $\mathbb{B}$ be two order $m$ dimension $n$ tensors. We say
that $\mathbb{A}$ and $\mathbb{B}$ are diagonal similar, if there exists some invertible diagonal matrix $D$ of order $n$ such that
$\mathbb{B}=D^{-(m-1)}\mathbb{A}D.$
\end{defn}
\begin{lem}{\rm (\cite{2013S})} \label{lem110}
Suppose that the two tensors $\mathbb{A}$ and $\mathbb{B}$ are diagonal
similar, namely $\mathbb{B}=D^{-(m-1)}\mathbb{A}D$ for some invertible diagonal matrix $D$.
Then $x$ is an eigenvector of  $\mathbb{B}$  corresponding to the eigenvalue $\lambda$
if and only if $y=Dx$ is an eigenvalues of $\mathbb{A}$ corresponding to the same eigenvalue $\lambda$.
\end{lem}

\section{Main result}
\hskip.6cm In this section, we will obtain the sharp upper and lower bounds for the spectral radius of a nonnegative weakly  irreducible tensor.
Applying this result to an irreducible matrix, we will obtain the main result  of   \cite{2017Y}.

For a tensor $\mathbb{A}$, for any $i\in [n]$,  we  denote $N_{\mathbb{A}}(i)$ (or simply $N(i)$) by
$$N_{\mathbb{A}}(i)=\{i_2, \ldots, i_m | a_{ii_2\ldots i_m}\not=0\}.$$

\begin{prop}\label{prop21}
 Let $\mathbb{A}= (a_{i_1i_2\ldots i_m})_{1\le i_j\le n \hskip.2cm (j=1, \ldots, m)}$ be a  nonnegative weakly
irreducible  tensor with order $m$ dimension $n$,
 Then for any $i\in [n]$, there exist some $j, k\in [n]\backslash \{i\}$ such that $j\in N(i)$ and $i\in N(k).$
\end{prop}
\begin{proof}
Let $I=\{i\}$.   For any $j\in [n]\backslash \{i\}$,  if  $j\not\in N(i)$, say,
 $a_{ii_2\ldots i_m}=0$ for $\forall i\in I$ and at least one of  $i_2, \ldots, i_m\not\in I$,
 then   $\mathbb{A}$ is weakly reducible, it is a contradiction.

 Similarly, let $I=[n]\backslash \{i\}$. For any $k\in I$,   if $i\not\in N(k)$, say,
 $a_{ki_2\ldots i_m}=0$ for $\forall k\in I$ and at least one of  $i_2, \ldots, i_m\not\in I$,
 then   $\mathbb{A}$ is weakly reducible, it is a contradiction.
\end{proof}

By Proposition \ref{prop21}, we obtain the following result easily.
\begin{prop}\label{prop22}
Let  $\mathbb{A}$  be a nonnegative weakly irreducibe tensor.  Then $r_i(\mathbb{A})>0$.
\end{prop}
\begin{them}\label{main result 1}
Let $\mathbb{A}= (a_{i_1i_2\ldots i_m})$ be a  nonnegative weakly
irreducible  tensor with order $m$ dimension $n$ and $a_{i\ldots i}=0$ for any $i\in [n]$.
Let  $N(i)=N_{\mathbb{A}}(i)$  defined as above, $R_i>0$ for any $i\in [n]$,
and $S_i=\sum\limits_{i_2, \ldots, i_m=1}^na_{ii_2\ldots i_m}R_{i_2}\ldots R_{i_m}$.
  Let $t_i\geq0$ and
  $\mathbb{B}=\mathbb{A}+\mathbb{M}$, where
$\mathbb{M}$ is a diagonal tensor with its diagonal element $m_{ii\ldots i }=t_{i}$.
 For any $1\leq i,j\leq n$, write $$F(i,j)=\frac{t_i+t_j+\sqrt{(t_i-t_j)^2+\frac{4S_iS_j}{(R_iR_j)^{m-1}}}}{2}.$$    Then
\begin{equation}\label{eq21}
\min\limits_{1\leq i,j\leq n}\{F(i,j), j\in N(i)\}\leq\rho(\mathbb{B})\leq \max\limits_{1\leq i,j\leq n}\{F(i,j), j\in N(i)\}.
\end{equation}
Moreover, one of the equalities in (\ref{eq21}) holds if and only if one of the two conditions holds:

\noindent {\rm (i) } $t_i+\frac{S_i}{R_i^{m-1}}=t_j+\frac{S_j}{R_j^{m-1}}$ for any $i,j\in[n]$.

\noindent  {\rm (ii) } There exist subsets $U$ and $W$ of $[n]$ such that

\noindent {\rm (1) } $[n]=U\cup W$ with $U\cap W=\phi$;

\noindent {\rm (2) }$a_{i_1i_2\ldots i_m}\neq0$  only when $i_1 \in  U, i_2,\ldots, i_m \in   W$ or
$i_1 \in W, i_2,\ldots, i_m \in  U$;

\noindent {\rm (3) } there exists $\ell>0$ such that
$\rho(\mathbb{B})=t_i+\frac{\ell^{m-1}S_{i}}{R_{i}^{m-1}}=t_j+\frac{S_j}{\ell^{m-1}R_j^{m-1}}$ for all $i\in U$  and all $j\in W.$
In fact, $\ell >1$ when the left equality holds and $\ell<1$ when the right equality holds.
\end{them}

\begin{Proof}
Let $R=diag(R_1,R_2,\ldots,R_n)$. Then $R$ is an invertible diagonal matrix, thus $\mathbb{B}$ and $R^{-(m-1)}\mathbb{B}R$ have the same spectra by Lemma \ref{lem110}.
For any $i \in [n]$, by Definition \ref{defn11}, we have
\begin{eqnarray}\begin{split}\label{eq22}
(R^{-(m-1)}\mathbb{B}R)_{ii_2\ldots i_m}
&=\sum\limits_{j=1}^{n} (R^{-(m-1)})_{ij}(\mathbb{B}R)_{ji_2\ldots i_m}\\
&=\sum\limits_{j=1}^{n} (R^{-(m-1)})_{ij}\sum\limits_{j_2, \ldots, j_m=1}^n b_{jj_2\ldots j_m}R_{j_2i_2}\ldots R_{j_mi_m}\\
&=(R^{-(m-1)})_{ii}b_{ii_2\ldots i_m}R_{i_2}\ldots R_{i_m}\\
&=R_i^{-(m-1)}b_{ii_2\ldots i_m}R_{i_2}\ldots R_{i_m}\\
&=\left\{\begin{array}{cc}
            t_i, & i_2=\ldots=i_m=i; \\
            R_i^{-(m-1)}a_{ii_2\ldots i_m}R_{i_2}\ldots R_{i_m}, & otherwise.
          \end{array}\right.
\end{split}\end{eqnarray}

\noindent Thus $R^{-(m-1)}\mathbb{B}R$ are also nonnegative weakly irreducible by Proposition \ref{prop15}.
By Lemma \ref{P-F theorem}, we know that there exists a positive eigenvector corresponding to
$\rho(\mathbb{B}),$ denoted by $X=(x_1,\ldots,x_n)^T,$
say, $(R^{-(m-1)}\mathbb{B}R)X=\rho(\mathbb{B}) X^{[m-1]}.$ Combining (\ref{eq22}), for any $i \in [n]$ we have
\begin{align}
((R^{-(m-1)}\mathbb{B}R)X)_i  &  =\sum\limits_{i_2, \ldots, i_m=1}^n(R^{-(m-1)}\mathbb{B}R)_{ii_2\ldots i_m}x_{i_2}\ldots x_{i_m} \nonumber \\
& =t_ix_i^{m-1}+\sum\limits_{i_{2},i_{3},\ldots, i_{m}=1}^{n}
R_i^{-(m-1)}R_{i_2}\ldots R_{i_m} a_{ii_2\ldots i_m}x_{i_2}\ldots
x_{i_m}\nonumber \\
&=\rho(\mathbb{B}) x_i^{m-1},\nonumber
\end{align}
and then \begin{eqnarray}\label{eq23}
(\rho(\mathbb{B})-t_i)x_i^{m-1}
=R_i^{-(m-1)}\sum\limits_{i_{2},i_{3},\ldots, i_{m}=1}^{n}R_{i_2}\ldots R_{i_m} a_{ii_2\ldots i_m}x_{i_2}\ldots x_{i_m}.
\end{eqnarray}

By (\ref{eq23}), Propositions \ref{prop21} and \ref{prop22},  we have $\rho(\mathbb{B})>t_i$ for any $i$.

\vskip.2cm
First we  prove the upper bounds for  $\rho(\mathbb{B})$.
Without loss of generality, we suppose $p,q\in [n]$ such that
$$x_p=\max\{x_i\,|\,i\in [n]\}=1, \quad  x_q=\max\{x_i\,|\,i\in N(p)\}.$$
By (\ref{eq23}), we have
\begin{eqnarray}\label{eq24}
\begin{split}
(\rho(\mathbb{B})-t_p)x_p^{m-1}&=R_p^{-(m-1)}\sum\limits_{i_{2},i_{3},\ldots,
i_{m}=1}^{n}R_{i_2}\ldots R_{i_m} a_{pi_2\ldots i_m}x_{i_2}\ldots x_{i_m}\\
&\leq R_p^{-(m-1)}\sum\limits_{i_{2},i_{3},\ldots, i_{m}=1}^{n}R_{i_2}\ldots
R_{i_m} a_{pi_2\ldots i_m}x_q^{m-1}\\
&= R_p^{-(m-1)}S_px_q^{m-1}
\end{split}
\end{eqnarray}
with equality if and only if Clause (a) holds, where  (a). $x_k=x_q$ for all $k\in N(p).$

Similarly, we have
\begin{eqnarray}\label{eq25}
\begin{aligned}
(\rho(\mathbb{B})-t_q)x_q^{m-1}&
=R_q^{-(m-1)}\sum\limits_{i_{2},i_{3},\ldots,i_{m}=1}^{n}R_{i_2}\ldots R_{i_m} a_{qi_2\ldots i_m}x_{i_2}\ldots x_{i_m}\\
&\leq R_q^{-(m-1)}\sum\limits_{i_{2},i_{3},\ldots, i_{m}=1}^{n}R_{i_2}\ldots
R_{i_m} a_{qi_2\ldots i_m}\\
&=R_q^{-(m-1)}S_q
\end{aligned}
\end{eqnarray}
with equality if and only if Clause (b) holds, where (b). $x_k=1$ for all  $k\in N(q).$

Therefore, by  (\ref{eq24})  and (\ref{eq25}), we have
$$(\rho(\mathbb{B})-t_p)(\rho(\mathbb{B})-t_q)\leq \frac{S_pS_q}{(R_pR_q)^{m-1}},$$
thus
\begin{equation}\label{eq26}
\rho(\mathbb{B})\leq\frac{t_p+t_q+\sqrt{(t_p-t_q)^2+\frac{4S_pS_q}{(R_pR_q)^{m-1}}}}{2}=F(p,q),
\end{equation}
and by $q\in N(p),$ we have
\begin{equation}\label{upper bound}
\rho(\mathbb{B})\leq \max\limits_{1\leq i,j\leq n}\{F(i,j), j\in N(i)\}.
\end{equation}

\vskip.2cm

Now we  prove the lower bounds for  $\rho(\mathbb{B})$.
Without loss of generality, we suppose $p,q\in [n]$ such that $$x_p=\min\{x_i\,|\,i\in [n]\}=1, \quad  x_q=\min\{x_i\,|\,i\in N(p)\}.$$ By (\ref{eq23}) we have
\begin{eqnarray}\label{eq28}
\begin{aligned}
\rho(\mathbb{B})-t_p&=(\rho(\mathbb{B})-t_p)x_p^{m-1}\\
&=R_p^{-(m-1)}\sum\limits_{i_{2},i_{3},\ldots, i_{m}=1}^{n}R_{i_2}\ldots R_{i_m}
a_{pi_2\ldots i_m}x_{i_2}\ldots x_{i_m}\\
&\geq (\frac{x_q}{R_p})^{m-1}\sum\limits_{i_{2},i_{3},\ldots,
i_{m}=1}^{n}R_{i_2}\ldots R_{i_m} a_{pi_2\ldots i_m}\\
&= (\frac{x_q}{R_p})^{m-1}S_p
\end{aligned}
\end{eqnarray}
with equality if and only if  $x_k=x_q$ for all  $k\in N(p).$

Similarly, we have
\begin{eqnarray}\label{eq29}
\begin{aligned}
(\rho(\mathbb{B})-t_q)x_q^{m-1}&=R_q^{-(m-1)}\sum\limits_{i_{2},i_{3},\ldots,
i_{m}=1}^{n}R_{i_2}\ldots R_{i_m} a_{qi_2\ldots i_m}x_{i_2}\ldots x_{i_m}\\
&\geq R_q^{-(m-1)}\sum\limits_{i_{2},i_{3},\ldots, i_{m}=1}^{n}R_{i_2}\ldots
R_{i_m} a_{qi_2\ldots i_m}\\
&= \frac{S_q}{R_q^{m-1}}
\end{aligned}
\end{eqnarray}
with equality if and only if $x_k=1$ for all $k\in N(q).$

By (\ref{eq28}) and (\ref{eq29}), we have
 $$(\rho(\mathbb{B})-t_p)(\rho(\mathbb{B})-t_q)\geq \frac{S_pS_q}{(R_pR_q)^{m-1}},$$
 thus
$$\rho(\mathbb{B})\geq\frac{t_p+t_q+\sqrt{(t_p-t_q)^2+\frac{4S_pS_q}{(R_pR_q)^{m-1}}}}{2}=F(p,q),$$\label{eq12}
and by $q\in N(p),$ we have
\begin{equation}\label{lower bound}
\rho(\mathbb{B})\geq\min\limits_{1\leq i,j\leq n}\{F(i,j), j\in N(i)\}.
\end{equation}

By (\ref{upper bound}) and (\ref{lower bound}), we complete the proof of (\ref{eq21}).

\vskip.2cm
Now we show the right equality in (\ref{eq21}) holds if and only if (i) or (ii) holds.
The proof of the left equality in (\ref{eq21}) is similar, we omit it.

Firstly, we complete the proof of the sufficiency part by the following two cases.

\noindent {\bf Case 1. } Condition (i) holds.

If for any $i,j\in[n]$, $t_i+\frac{S_i}{R_i^{m-1}}=t_j+\frac{S_j}{R_j^{m-1}}$,
then $t_i-t_j=\frac{S_j}{R_j^{m-1}}-\frac{S_i}{R_i^{m-1}},$
and
$$F(i,j)=\frac{t_i+t_j+\sqrt{(t_i-t_j)^2+\frac{4S_iS_j}{(R_iR_j)^{m-1}}}}{2}=t_i+\frac{S_i}{R_i^{m-1}}.$$
Thus $\max\limits_{1\leq i,j\leq n}\{F(i,j),j\in N(i)\}=t_i+\frac{S_i}{R_i^{m-1}}.$

On the other hand, by (\ref{eq22}), $R^{-(m-1)}\mathbb{B}R$ have the same row sum
$t_i+\frac{S_i}{R_i^{m-1}},$  thus   by Lemma \ref{lem18} and $R^{-(m-1)}\mathbb{B}R$  is  nonnegative weakly
irreducible, we have

$$\rho(\mathbb{B})=\rho(R^{-(m-1)}\mathbb{B}R)=t_i+\frac{S_i}{R_i^{m-1}}=\max\limits_{1\leq
i,j\leq n}\{F(i,j),j\in N(i)\}.$$

 \noindent{\bf Case 2.} Condition (ii) holds.

There exist nonempty proper subsets $U$ and $W$ of $[n]$ such that
$a_{i_1i_2\ldots i_m}\neq0$ only when $i_1\in U,
i_2,\ldots,i_m\in  W$  or
$i_1\in  W, i_2,\ldots,i_m\in U.$
Let $\alpha=t_i+(\frac{\ell}{R_i})^{m-1}S_i=t_{j}+\frac{S_{j}}{(\ell R_{j})^{m-1}}$
for all $i\in U$ and all $j\in W$.
Construct a positive vector $Y=(y_1,y_2,\ldots, y_n)^T$ with $y_i=R_i$ when $i\in U$ and $y_i= \ell R_i$ when
                                                  $i\in W.$  We will check $\mathbb{B}Y=\alpha Y^{[m-1]}.$
In fact when $i\in U$ we have
\begin{eqnarray*}
(\mathbb{B}Y)_i &=&\sum\limits_{i_{2},i_{3},\ldots, i_{m}=1}^{n}b_{ii_2\ldots i_m}y_{i_2}\ldots y_{i_m} \\
&=&b_{ii\ldots i}R_i^{m-1}+\sum\limits_{i_{2},i_{3},\ldots, i_{m}\in W}a_{ii_2\ldots i_m}\ell R_{i_2}\ldots \ell R_{i_m}\\
&=&t_iR_i^{m-1}+\ell^{m-1}S_i\\
&=&\Big[t_i+(\frac{\ell}{R_i})^{m-1}S_i \Big]R_i^{m-1}\\
&=&\alpha y_i^{m-1},
\end{eqnarray*}
and for any  $j\in W$,
\begin{eqnarray*}
(\mathbb{B}Y)_j
&=&\sum\limits_{i_{2},i_{3},\ldots, i_{m}=1}^{n}b_{ji_2\ldots i_m}y_{i_2}\ldots y_{i_m}\\
&=&b_{jj\ldots j}(\ell R_{j})^{m-1}+\sum\limits_{i_{2},i_{3},\ldots, i_{m}\in U} a_{j i_2\ldots i_m}R_{i_2}\ldots R_{i_m}\\
&=&t_{j}(\ell R_{j})^{m-1}+S_{j}.\\
&=&\Big[t_{j}+\frac{S_{j}}{(\ell R_{j})^{m-1}}\Big]( \ell R_j)^{m-1}\\
&=&\alpha y_j^{m-1}.
\end{eqnarray*}

 Then $\alpha$ is an eigenvalue of $B$ with eigenvector $Y$,  and thus $\rho(\mathbb{B})=\alpha $ by  $Y$ is a positive vector and Lemma \ref{P-F theorem}.

 On the other hand, if $j\in N(i),$ then

$$t_i-t_j=\left\{\begin{array}{cc}
                   \frac{S_j}{(\ell R_j)^{m-1}}-(\frac{\ell}{R_i})^{m-1}S_i, &\mbox{ if } i\in U, j\in  W; \\
                  (\frac{\ell}{R_j})^{m-1}S_j-\frac{S_i}{(\ell R_i)^{m-1}}, &\mbox{ if }  i\in  W,j\in U.
                 \end{array}\right.$$

\noindent and

$F(i,j)=\frac{t_i+t_j+\sqrt{(t_i-t_j)^2+\frac{4S_iS_j}{(R_iR_j)^{m-1}}}}{2}$

\hskip1.2cm $=\left\{\begin{array}{ll}
\frac{t_i+t_j+\frac{S_j}{(\ell R_j)^{m-1}}+(\frac{\ell}{R_i})^{m-1}S_i}{2}, & \mbox{ if } i\in U, j\in  W; \\
 \frac{ t_i+t_j+\frac{S_i}{(\ell R_i)^{m-1}}+(\frac{\ell}{R_j})^{m-1}S_j}{2}, &\mbox{ if }  i\in  W,j\in U.
 \end{array}\right.$

\hskip1.2cm  $=\alpha,$
\vskip.2cm

\noindent thus we have $\max\limits_{1\leq i,j\leq n}\{F(i,j),j\in N(i)\}=\alpha.$

Combining the above arguments, we have $\rho(\mathbb{B})=\max\limits_{1\leq i,j\leq n}\{F(i,j),j\in N(i)\}.$

\vskip.2cm
Now we focus on the necessity part.

Let $p,q\in [n]$ which defined in the proof of the upper bounds in (2.1), i.e.,
$x_{p}=\max \{x_{i}\, |\, i\in [n]\}=1, x_{q}=\max \{x_{i}\, |\, i\in N(p)\}$ and $\rho(\mathbb{B})\leq F(p,q)$.
Now $\rho(\mathbb{B})=\max\limits_{1\leq i,j\leq n}\{F(i,j),j\in N(i)\}\geq F(p,q)$, so $\rho(\mathbb{B})= F(p,q)$, which means the equality in (\ref{eq26}) occurs. Thus the equalities in (\ref{eq24}) and (\ref{eq25}) both hold.  If we write $T(i)=R_i^{-(m-1)}S_i$  for any $i\in[n]$, then we have
           $$\rho(\mathbb{B})-t_p=T(p)x_q^{m-1}, \quad  \rho(\mathbb{B})-t_q=T(q)x_q^{-(m-1)}.$$
Write $p_{1}=p$ and $q_{1}=q$. By Clause (a) we have $x_i=x_q$ for any $i \in N(p_{1})$, and by Clause (b) we have $x_{i}=1$ for any $i \in N(q_{1})$.
Pick $p_{2}$ in $N(q_{1})$ and let $q_{2} \in N(p_{2}) $ such that $x_{q_{2}}=\max \{x_{i}\, |\, i\in N(p_{2})\}$.
Using the similar arguments for the pair $(p_{2}, q_{2})$ as that of the above pair $(p,q)$,
we have $$\rho(\mathbb{B})-t_{p_{2}}=T(p_{2})x_{q_{2}}^{m-1}, \quad          \rho(\mathbb{B})-t_{q_{2}}=T(q_{2})x_{q_{2}}^{-(m-1)}.$$
And we have $x_i=x_{q_{2}}$ for any $i \in N(p_{2})$ and $x_{i}=1$ for any $i \in N(q_{2})$.

Pick $p_{3}$ in $N(q_{2})$ and let $q_{3} \in N(p_{3}) $ and repeat the above arguments,
then we may obtain a sequence $p_{1}, q_{1},p_{2}, q_{2},p_{3}, q_{3},\ldots$,
where  $$\rho(\mathbb{B})-t_{p_{i}}=T(p_{i})x_{q_{i}}^{m-1}, \quad \rho(\mathbb{B})-t_{q_{i}}=T(q_{i})x_{q_{i}}^{-(m-1)},$$
and  $x_{p_{i}}=1$, $x_{q_{i}}=\max \{x_{i}\, |\, i\in N(p_{i})\}$.

Now we will prove all $x_{q_{i}}$'s are equal. First we will prove  $x_{q_{i+1}}\geq x_{q_{i}}$. Combining
 $$\rho(\mathbb{B})-t_{q_{i}}=T(q_{i})x_{q_{i}}^{-(m-1)}, \quad \rho(\mathbb{B})-t_{p_{i+1}}=T(p_{i+1})x_{q_{i+1}}^{m-1},$$ we have
$$\rho(\mathbb{B})=\frac{t_{q_{i}}+t_{p_{i+1}}+
\sqrt{(t_{q_{i}}-t_{p_{i+1}})^2+(\frac{x_{q_{i+1}}}{x_{q_{i}}})^{m-1}4T(q_{i})T(p_{i+1})}}{2}.$$

 On the other hand, we have
$$\rho(\mathbb{B})\geq  F(q_{i},p_{i+1})=\frac{t_{q_{i}}+t_{p_{i+1}}+\sqrt{(t_{q_{i}}-t_{p_{i+1}})^2+4T(q_{i})T(p_{i+1})}}{2}.$$
So we have $x_{q_{i+1}}\geq  x_{q_{i}}$ for any $i$ by the above two inequalities.

Consider the associated directed graph of $G(\mathbb{B})$, $D(G(\mathbb{B}))$, by Proposition \ref{prop15},
$D(G(\mathbb{B}))$ is strongly connected since $\mathbb{A}$ thus $\mathbb{B}$ is weakly irreducible.
So for any $q_{i}$ and $q_{j}$,  there is a path from $q_{i}$ to $q_{j}$ in $D(G(\mathbb{B}))$,
so we have  $x_{q_{j}}\geq x_{q_{i}}$. On the other hand, there is also a path from $q_{j}$ to $q_{i}$,
so $x_{q_{i}}\geq x_{q_{j}}$. Thus  we have showed $x_{q_{j}}= x_{q_{i}}$.

Now we define the set  $U$ which contains $p_{i}^{\prime s}$, where $p_{i}=p_{1}$ or $p_{i}$ is in some $N(q_{j})$.
Define the set  $W$ which contains $q_{i}^{\prime s}$, where  $q_{i}$ is in some $N(p_{j})$.
We have proved that $x_{p_{i}}=1$ for any $p_{i}$ in $U$, and $x_{q_{i}}= x_{q}$ for any $q_{i}$ in $W$.

In the following, we will prove that $U \cup W=[n]$. Suppose to the contrary that some $k \not \in U \cup W.$
There is a directed path from $p$ to $k$ in the strongly connected directed graph $D(G(\mathbb{B}))$, say $pk_{1}\ldots k_{s-1}k_{s}k$.
It is obvious that $k \not \in U $ implies that $k_{s}\not \in W$, while $k \not \in W $ implies that $k_{s}\not \in U$ either.
Hence $k_{s}\not \in U \cup W.$ By using this arguments we conclude that $k_{s-1}\not \in U \cup W.$
And then $p\not \in U \cup W,$  which is a contradiction.
Now we distinguish two cases to finish the proof.

\noindent {\bf Case 1: }$x_q=1$.

In this case,  we will show Condition (i) holds, i.e., $t_i+\frac{S_i}{R_i^{m-1}}=t_j+\frac{S_j}{R_j^{m-1}}$ holds for all $i,j\in[n].$  If $x_q=1$, then $X=(1,1,\ldots,1)^T.$  Therefore, by (\ref{eq23}) we have
$\rho(\mathbb{B})=t_i+\frac{S_i}{R_i^{m-1}}$ holds for any $i\in[n].$

\noindent {\bf Case 2: }$x_q<1$.

In this case, we will show Condition (ii) holds.
Take $\ell=x_q$, then $0<\ell<1$, and then for the eigenvector $X$ we have  $x_i=1$ if $i\in U$ and  $x_i=\ell$ when $i\in W$.
From the above arguments,  we know  $U\cap W=\phi$ and
$a_{i_1i_2\ldots i_m}\neq 0,$ if $i_1\in U$, then $i_2,\ldots,i_m\in W$, or $i_1\in W$ and then $i_2\ldots i_m\in U$.

By using (\ref{eq23}) for any $i\in U$ ($x_i=1$), then we have
\begin{eqnarray*}
\rho(\mathbb{B})&=&
t_ix_i^{m-1}+R_{i}^{-(m-1)} \sum\limits_{i_{2},i_{3},\ldots, i_{m}=1}^{n} a_{ii_2\ldots
i_m}x_{i_2}\ldots x_{i_m}R_{i_2}\ldots R_{i_m}\\
&=&t_i+R_{i}^{-(m-1)}\ell^{m-1}S_{i},
\end{eqnarray*}
For any $j\in W$ ($x_j=\ell$) we have
\begin{eqnarray*}
\rho(\mathbb{B})x_j^{m-1}&=&
t_jx_j^{m-1}+R_{j}^{-(m-1)}\sum \limits_{i_{2},i_{3},\ldots, i_{m}=1}^{n} a_{ji_2\ldots
i_m}x_{i_2}\ldots x_{i_m}R_{i_2}\ldots R_{i_{m}}\\
&=&t_jx_j^{m-1}+R_{j}^{-(m-1)}S_j,
\end{eqnarray*}
and then $\rho(\mathbb{B})=t_j+R_{j}^{-(m-1)}S_j \ell^{-(m-1)}.$
Hence for all $i\in U$  and all $j\in W$ we have
$$\rho(\mathbb{B})=t_i+\frac{\ell^{m-1}S_{i}}{R_{i}^{m-1}}=t_j+\frac{S_j}{\ell^{m-1}R_j^{m-1}}.$$
We finish the proof.
\end{Proof}

Particularly, if we  define $R=diag(r_1,r_2,\ldots, r_n)$ where $r_i=r_i(\mathbb{A})$, and
$s_i=\sum\limits_{i_2, \ldots, i_m=1}^na_{ii_2\ldots i_m}r_{i_2}\ldots r_{i_m}$ in Theorem \ref{main result 1}, we may obtain
the following result.

\begin{cor}\label{cor24}
Let $\mathbb{A}= (a_{i_1i_2\ldots i_m})$ be a  nonnegative weakly
irreducible  tensor with order $m$ dimension $n$ and $a_{i\ldots i}=0$ for any $i\in [n]$.
Let $r_i=r_i(\mathbb{A})$ and $N(i)=N_{\mathbb{A}}(i)$  defined as above,
and $s_i=\sum\limits_{i_2, \ldots, i_m=1}^na_{ii_2\ldots i_m}r_{i_2}\ldots r_{i_m}$.
  Let $t_i\geq0$ and $\mathbb{B}=\mathbb{A}+\mathbb{M}$, where
$\mathbb{M}$ is a diagonal tensor with its diagonal element $m_{ii\ldots i }=t_{i}$.
 For any $1\leq i,j\leq n$, write $$f(i,j)=\frac{t_i+t_j+\sqrt{(t_i-t_j)^2+\frac{4s_is_j}{(r_ir_j)^{m-1}}}}{2}.$$    Then
\begin{equation}\label{eq213}
\min\limits_{1\leq i,j\leq n}\{f(i,j), j\in N(i)\}\leq\rho(\mathbb{B})\leq \max\limits_{1\leq i,j\leq n}\{f(i,j), j\in N(i)\}.
\end{equation}
Moreover, one of the equalities in (\ref{eq213}) holds if and only if one of the two conditions holds:

\noindent {\rm (i) } $t_i+\frac{s_i}{r_i^{m-1}}=t_j+\frac{s_j}{r_j^{m-1}}$ for any $i,j\in[n]$.

\noindent  {\rm (ii) } There exist nonempty proper subsets $U$ and $W$ of $[n]$ such that

\noindent {\rm (1) } $[n]=U\cup W$ with $U\cap W=\phi$;

\noindent {\rm (2) }$a_{i_1i_2\ldots i_m}\neq0$  only when $i_1 \in  U, i_2,\ldots, i_m \in   W$ or
$i_1 \in W, i_2,\ldots, i_m \in  U$;

\noindent {\rm (3) } there exists $\ell>0$ such that
$\rho(\mathbb{B})=t_i+\frac{\ell^{m-1}s_{i}}{r_{i}^{m-1}}=t_j+\frac{s_j}{\ell^{m-1}r_j^{m-1}}$ for all $i\in U$  and all $j\in W.$
In fact, $\ell >1$ when the left equality holds and $\ell<1$ when the right equality holds.
\end{cor}

Let $m=2$, we have the following results on the matrix case by Corollary \ref{cor24}.
\begin{cor}\label{cor25}{\rm (\cite{2017Y},Theorem 2.2)}
Let $A=(a_{ij})$ be an $n\times n$  nonnegative  irreducible matrix with $a_{ii}=0$ for  $i=1,2,\ldots,n$, and the row sum $r_1,r_2,\ldots,r_n$.
Let $B=A+M$, where $M=diag(t_1,t_2,\ldots,t_n)$ with $t_i\geq0$ for any $i\in\{1,2,\ldots,n\}$, $s_i=\sum\limits_{j=1}^n{a_{ij}r_j}$,
$\rho(B)$ be the spectral radius of $B$.
Let $f(i,j)=\frac{t_i+t_j+\sqrt{(t_i-t_j)^2+\frac{4s_is_j}{r_ir_j}}}{2}$ for any $1\leq i,j\leq n$.
Then
\begin{equation}\label{eq214}
\min\limits_{1\leq i,j\leq n}\{f(i,j), a_{ij}\not=0\}\leq \rho(B)
\leq \max\limits_{1\leq i,j\leq n}\{f(i,j), a_{ij}\not=0\}.
\end{equation}
Moreover, one of the equalities  in (\ref{eq214}) holds if and only if one of the two conditions holds:

\noindent {\rm (i) } $t_i+\frac{s_i}{r_i}=t_j+\frac{s_j}{r_j}$ for any $i,j\in\{1,2,\ldots,n\}$;

\noindent  {\rm (ii) } There exist subsets $U$ and $W$ of $[n]$ such that

\noindent {\rm (1) } $[n]=U\cup W$ with $U\cap W=\phi$;

\noindent {\rm (2) }$a_{ij}\neq0$  only when $i \in  U, j\in   W$ or $i\in W, j\in  U$;

\noindent {\rm (3) } there exists $\ell>0$ such that
$\rho(\mathbb{B})=t_i+\frac{\ell s_{i}}{r_{i}}=t_j+\frac{s_j}{\ell r_j}$ for all $i\in U$  and all $j\in W.$
In fact, $\ell >1$ when the left equality holds and $\ell<1$ when the right equality holds.
\end{cor}

\section{Applications to a $k$-uniform hypergraph}
\hskip.6cm It is well known that a hypergraph is a natural generalization of an ordinary graph {\rm (\cite{1973B})}.
A hypergraph $\mathcal{H}=(V(\mathcal{H}), E(\mathcal{H}))$ on $n$ vertices is a set of vertices, say, $V(\mathcal{H})=\{1,2,\ldots,n\}$
and a set of edges, say, $E(\mathcal{H})=\{e_1,e_2,\ldots,e_m\},$ where $e_i=\{i_1,i_2\ldots,i_l\}, i_j\in [n], j=1,2,\ldots,l.$
Let $k\geq 2$, if $\mid e_i\mid=k$ for any $i=1,2,\ldots ,m$, then $\mathcal{H}$ is called a $k$-uniform hypergraph.
Especially, if $k=2$, then $\mathcal{H}$ is an ordinary graph.
The degree $d_i$ of vertex $i$ is defined as $d_i=|\{e_j:i\in e_j\in E(\mathcal{H})\}|.$
If $d_i=d$ for any vertex $i$ of a hypergraph $\mathcal{H}$, then $\mathcal{H}$ is called $d$-regular.
A {\it walk} $W$ of length $\ell$ in $\mathcal{H}$ is a sequence of alternate vertices and edges: $v_0,e_1,v_1,e_2,\cdots,e_\ell,v_\ell,$ where $\{v_i,v_{i+1}\}\subseteq e_{i+1}$ for $i=0,1,\cdots,\ell-1$. The hypergraph $\mathcal{H}$ is said to be connected if every two vertices are connected by a walk.

\begin{defn}{\rm (\cite{Cooper2012} \cite{2014Qi})}
Let $\mathcal{H}=(V(\mathcal{H}),E(\mathcal{H}))$ be a $k$-uniform hypergraph on $n$ vertices.
The adjacency tensor of $\mathcal{H}$ is defined as the order $k$ dimension $n$ tensor $\mathbb{A}(\mathcal{H})$,
 whose $(i_1i_2\ldots i_k)$-entry is
$$\mathbb{A}(\mathcal{H})_{i_1i_2\ldots i_k}=\left\{\begin{array}{cc}
                                             \frac{1}{(k-1)!}, & \mbox{ if }
                                             \{i_1,i_2,\ldots,i_k\}\in
                                             E(\mathcal{H}), \\
                                             0, & \mbox{ otherwise} .
                                              \end{array}\right.$$
\end{defn}
Let $\mathbb{D}(\mathcal{H})$ be an order $k$ dimension $n$ diagonal tensor with its diagonal entry $\mathbb{D}_{ii\ldots i}$ being $d_i$,
the degree of vertex $i$, for all $i\in V(\mathcal{H})=[n]$.
Then $\mathbb{Q}(\mathcal{H})=\mathbb{D(\mathcal{H})}+ \mathbb{A(\mathcal{H})}  $ is the signless Laplacian tensor of the hypergraph $\mathcal{H}$.

Clearly, the adjacency tensor and the signless Laplacian tensor of a  hypergraph are nonnegative.
It was proved in \cite{Friedland2013} that a $k$-uniform hypergraph $\mathcal{H}$
is connected if and only if its adjacency tensor $\mathbb{A}(\mathcal{H})$
(and thus the signless Laplacian tensor $\mathbb{Q}(\mathcal{H})$)  is weakly irreducible.
For any vertices $i\in [n]$ of a $k$-uniform hypergraph $\mathcal{H}$,
we take
$$m_i=\frac{\sum\limits_{\{i,i_2, \ldots, i_k\}\in E(\mathcal{H})}d_{i_2}\ldots
d_{i_k}}{d_i^{k-1}},$$
which is a generalization of the average of degrees of
vertices adjacent to $i$ of the ordinary graph.

Recently, several papers studied the spectral radii of the adjacency tensor $\mathbb{A}(\mathcal{H})$
and the signless Laplacian tensor $\mathbb{Q}(\mathcal{H})$
of a $k$-uniform hypergraph $\mathcal{H}$ (see  \cite{2016L, 2015Y} and so on). In this section,
we will apply Theorem \ref{main result 1} 
to the adjacency tensor $\mathbb{A}(\mathcal{H})$
and the signless Laplacian tensor $\mathbb{Q}(\mathcal{H})$ of a $k$-uniform hypergraph $\mathcal{H}$. Some known   and new results about the bounds of $\rho(\mathbb{A}(\mathcal{H}))$ and $\rho(\mathbb{Q}(\mathcal{H}))$ will show.

\begin{them}\label{thm32}
Let $k\geq 3$,  $b_i>0$ for any $i\in [n]$,  and $\mathcal{H}$ be a connected $k$-uniform hypergraph on $n$ vertices. Then
\begin{equation}\label{eq31}
 \min\limits_{e\in E(\mathcal{H})}\min\limits_{\{i,j\}\subseteq e}\sqrt{b^{\prime}_ib^{\prime}_j}\leq \rho(\mathbb{A(\mathcal{H})})\leq
 \max\limits_{e\in E(\mathcal{H})}\max\limits_{\{i,j\}\subseteq e}\sqrt{b^{\prime}_ib^{\prime}_j},
\end{equation}
here for any $i\in [n]$,
$$b^{\prime}_i=b^{-(k-1)}_i\sum\limits_{\{i,i_2,\ldots, i_k\}\in E(\mathcal{H})}b_{i_2}\ldots b_{i_k}.$$
 Moreover, one of the equalities in (\ref{eq31}) holds if and
only if $b^{\prime}_i=b^{\prime}_j$ for any $i,j \in[n]$.
\end{them}

\begin{Proof}
We apply Theorem  \ref{main result 1} to $\mathbb{A}(\mathcal{H})$
and take $R=diag(b_1,b_2,\ldots, b_n)$.
Let $\mathbb{A}=\mathbb{B}=\mathbb{A}(\mathcal{H})$.
Then $t_i=0, a_{ii\ldots i}=0$, $R_i=b_i$ for any $i\in [n]$,
 and $$a_{i_1i_2\ldots i_k}=\left\{\begin{array}{cc}
                                             \frac{1}{(k-1)!}, & \mbox{ if }
                                             \{i_1,i_2,\ldots,i_k\}\in
                                             E(\mathcal{H}); \\
                                             0, & \mbox{ otherwise}.
                                              \end{array}\right.$$
If $a_{ii_2\ldots i_k}\not=0$, then there are $(k-1)!$ entries $a_{is_2\ldots s_k}\not=0$ in $\mathbb{A}(\mathcal{H})$,
where  $s_2\ldots s_k$ is a permutation of $i_2\ldots i_k$. Thus  for any $j\in N(i)$, say, for any $\{i,j\}\subseteq e\in E(\mathcal{H})$, we have
\begin{eqnarray}\begin{split}\label{eq32}
 F(i,j)&=\sqrt{\frac{S_iS_j}{(R_iR_j)^{k-1}}}\\
 &=\sqrt{\frac{\sum\limits_{i_{2},i_{3},\ldots, i_{k}=1}^{n} a_{ii_2\ldots
 i_k}b_{i_2}\cdots b_{i_k}\sum\limits_{j_{2},j_{3},\ldots, j_{k}=1}^{n} a_{jj_2\ldots j_k}b_{j_2}\cdots b_{j_k}}{(b_ib_j)^{k-1}}}\\
&=\sqrt{b^{\prime}_ib^{\prime}_j}.
\end{split}
\end{eqnarray}
Therefore (\ref{eq31}) holds by (\ref{eq213}) and (\ref{eq32}).

Furthermore, $t_i+\frac{S_i}{R_i^{k-1}}=t_j+\frac{S_j}{R_j^{k-1}}$ for any
$i,j\in[n]$ implies $b^{\prime}_i=b^{\prime}_j$  for any $i,j\in[n]$ by the definitions of $S_i$ and $b^{\prime}_i$.

We note that the adjacency tensor of any $k$-uniform hypergraph is a symmetric tensor, say,
 $a_{i_1i_2\ldots i_k}\neq0$ implies  $a_{j_1j_2\ldots j_k}\neq0$ where $j_1j_2\ldots j_k$ is a permutation of $i_1i_2\ldots i_k$.
  But by (ii) of Theorem  \ref{main result 1}, if $a_{i_1i_2\ldots i_k}\neq0$ implies $i_1\in U$, $i_2,\ldots, i_k\in W$ or the vice,
  and then  $a_{i_2i_1i_3\ldots i_k}=0$ by $k\geq 3$, it is a contradiction.

Combining the above arguments,
we know one of the equalities in (\ref{eq31}) holds if and only if $b^{\prime}_i=b^{\prime}_j$ for any $i,j \in[n]$ by Theorem  \ref{main result 1}.
\end{Proof}

\begin{cor}\label{cor33}
Let $k\geq 3$ and $\mathcal{H}$ be a connected $k$-uniform hypergraph on $n$ vertices. Then
\begin{equation}\label{eq33}
 \min\limits_{e\in E(\mathcal{H})}\min\limits_{\{i,j\}\subseteq
e}\sqrt{m_im_j}\leq \rho(\mathbb{A(\mathcal{H})})\leq \max\limits_{e\in
E(\mathcal{H})}\max\limits_{\{i,j\}\subseteq e}\sqrt{m_im_j}.
\end{equation}  Moreover, one of the equalities in (\ref{eq33}) holds if and
only if $m_i=m_j $ for any $i,j \in[n]$
\end{cor}
\begin{Proof}
We apply Theorem \ref{thm32} to $\mathbb{A}(\mathcal{H})$.  Let  $b_i=d_i$ for any $i\in [n]$.
Then $b^{\prime}_i=m_i$ for any $i\in[n]$ and thus the results hold by Theorem \ref{thm32}.
\end{Proof}

\begin{cor}\label{cor34}
Let $k\geq3$ and $\mathcal{H}$ be a connected $k$-uniform hypergraph on $n$ vertices. Then
\begin{equation}\label{eq34}
\min\limits_{e\in E(\mathcal{H})}\min\limits_{\{i,j\}\subseteq
e}\sqrt{d_id_j}\leq \rho(\mathbb{A(\mathcal{H})})\leq \max\limits_{e\in
E(\mathcal{H})}\max\limits_{\{i,j\}\subseteq e}\sqrt{d_id_j}.
\end{equation}
 Moreover, one of the equalities in (\ref{eq34}) holds if and only if $\mathcal{H}$ is a regular hypergraph.
\end{cor}
\begin{Proof}
We apply Theorem  \ref{thm32} to $\mathbb{A}(\mathcal{H})$ and take $b_i=1$ for each $i\in [n]$.
 Then $b^{\prime}_i=d_i$ for any $i\in[n]$ and thus the results hold by Theorem \ref{thm32}.
\end{Proof}

\begin{rem}\label{rem35}
The right inequality in Corollary \ref{cor33} and Corollary \ref{cor34} is the result of Remark 14 in \cite{2015Y}.
\end{rem}

\begin{them}\label{thm36}
Let $k\geq 3$,  $b_i>0$ for any $i\in [n]$,  and $\mathcal{H}$ be a connected $k$-uniform hypergraph on $n$ vertices.
Then
\begin{equation}\label{eq35}
\min\limits_{e\in E(\mathcal{H})}\min\limits_{\{i,j\}\subseteq e}g(i,j)\leq\rho(\mathbb{Q(\mathcal{H})})\leq\max\limits_{e\in E(\mathcal{H})}\max\limits_{\{i,j\}\subseteq e}g(i,j),
\end{equation}
where $$g(i,j)=\frac{d_i+d_j+\sqrt{(d_i-d_j)^2+4b^{\prime}_ib^{\prime}_j}}{2} ,$$
$$b^{\prime}_i=b^{-(k-1)}_i\sum\limits_{\{i,i_2,\ldots, i_k\}\in E(\mathcal{H})}b_{i_2}\ldots b_{i_k}.$$
Moreover, one of the equalities in (\ref{eq35}) holds if and
only if $d_i+b^{\prime}_i=d_j+b^{\prime}_j$ for any $i,j \in[n]$.
\end{them}

\begin{Proof}
We apply Theorem  \ref{main result 1} to $\mathbb{Q}(\mathcal{H})$
and take $R=diag(b_1,b_2,\cdots,b_n)$.
 Let $\mathbb{A}=\mathbb{A}(\mathcal{H})$ and $\mathbb{B}=\mathbb{Q}(\mathcal{H})$.
Then $t_i=d_i, a_{ii\ldots i}=0$, $R_i=b_i$ for any $i\in [n]$.
Thus for any $\{i,j\}\subseteq e\in E(\mathcal{H})$, we have
\begin{eqnarray}\begin{split}\label{eq36}
 \frac{t_i+t_j +\sqrt{(t_i-t_j)^2+\frac{4S_iS_j}{(R_iR_j)^{m-1}}}}{2}
 =\frac{d_i+d_j+\sqrt{(d_i-d_j)^2+4b^{\prime}_ib^{\prime}_j}}{2}.
\end{split}
\end{eqnarray}
Therefore (\ref{eq35}) holds by (\ref{eq213}) and (\ref{eq36}).

Furthermore, $t_i+\frac{S_i}{R_i^{k-1}}=t_j+\frac{S_j}{R_j^{k-1}}$ for any
$i,j\in[n]$ implies $d_i+b^{\prime}_i=d_j+b^{\prime}_j$  for any $i,j\in[n]$ by the definition of $S_i$ and $b^{\prime}_i$.

Similar to the proof of Theorem \ref{thm32}, we know the condition (ii) of Theorem  \ref{main result 1} will not hold.
Thus one of the equalities in (\ref{eq35}) holds if and only if $d_i+b^{\prime}_i=d_j+b^{\prime}_j$ for any $i,j \in[n]$ by Theorem  \ref{main result 1}.
\end{Proof}

\begin{rem}\label{rem37}
The right inequality in Theorem \ref{thm36}  is the result of Theorem 11 in \cite{2015Y}.
\end{rem}

\begin{cor}\label{cor38}
Let $k\geq3 $ and $\mathcal{H}$ be a connected $k$-uniform hypergraph on $n$ vertices. Then
\begin{equation}\label{eq37}
 \min\limits_{e\in E(\mathcal{H})}\min\limits_{\{i,j\}\subseteq e}h(i,j)\leq
\rho(\mathbb{Q(\mathcal{H})})\leq \max\limits_{e\in
E(\mathcal{H})}\max\limits_{\{i,j\}\subseteq e}h(i,j),
\end{equation}
where
 $$h(i,j)=\frac{d_i+d_j+\sqrt{(d_i-d_j)^2+4m_im_j}}{2}.$$
Moreover, one of the equalities in (\ref{eq37}) holds if and only if $d_i+m_i=d_j+m_j$ for all $i,j\in [n]$.
\end{cor}

\begin{Proof}
We apply Theorem \ref{thm36} to $\mathbb{Q}(\mathcal{H}).$  Let $b_i=d_i$ for any $i\in [n]$.
Then $b^{\prime}_i=m_i$ for any $i\in[n]$ and thus the results hold by Theorem \ref{thm36}.
\end{Proof}

\begin{rem}\label{rem39}
The right inequality in Corollary \ref{cor38}  is the result of Corollary 13 in \cite{2015Y}.
\end{rem}

\begin{cor}\label{cor310}
Let $k\geq3$ and $\mathcal{H}$ be a connected $k$-uniform hypergraph on $n$ vertices. Then
\begin{equation}\label{eq38}
 \min\limits_{e\in E(\mathcal{H})}\min\limits_{\{i,j\}\subseteq e}(d_i+d_j)\leq
\rho(\mathbb{Q(\mathcal{H})})\leq \max\limits_{e\in
E(\mathcal{H})}\max\limits_{\{i,j\}\subseteq e}(d_i+d_j).
\end{equation}
Moreover, one of the equalities in (\ref{eq38}) holds if and only if $\mathcal{H}$ is a regular hypergraph. 
\end{cor}

\begin{Proof}
We apply Theorem  \ref{thm36} to $\mathbb{Q}(\mathcal{H}),$ and take $b_i=1$ for each $i\in [n]$.
 Then $b^{\prime}_i=d_i$ for any $i\in[n]$ and thus the results hold by Theorem \ref{thm36}.
\end{Proof}

\begin{rem}\label{rem311}
The right inequality in Corollary \ref{cor310}  is the result of Corollary 12 in \cite{2015Y}.
\end{rem}

\end{document}